\documentclass[12pt,A4]{article}

\DeclareSymbolFont{AMSb}{U}{msb}{m}{n}
\DeclareSymbolFontAlphabet{\Bbb}{AMSb}
\usepackage{latexsym}
\usepackage[dvips]{graphicx}
\usepackage{amstext,amsfonts,amsmath,amsthm,graphicx,amssymb,amscd,epsfig}
\usepackage{enumerate}
\usepackage{psfrag} 
\usepackage[all,2cell,ps]{xy}
\usepackage{pstricks,pst-node,pst-coil,pst-plot}

\input epsf
\newtheorem{teo}{Theorem}[section]

\newtheorem{lem}[teo]{Lemma}
\newtheorem{cor}[teo]{Corollary}
\newtheorem{prop}[teo]{Proposition}
\newtheorem{rem}[teo]{Remark}

\newcommand{\norm}[1]{\left\Vert\,#1\right\Vert}

\newcommand{\C}{\mathbb{C}}
\newcommand{\R}{\mathbb{R}}

\newcommand{\N}{\mathbb{N}}
\newcommand{\NN}{\mathcal{N}}

\def\be{\begin{equation}}
\def\ee{\end{equation}}
\def\bq{\begin{eqnarray}}
\def\eq{\end{eqnarray}}
\def\beq{\begin{eqnarray}}
\def\eeq{\end{eqnarray}}
\def\ba{\begin{array}}
\def\ea{\end{array}}
\def\bi{\begin{itemize}}
\def\ei{\end{itemize}}

\newcommand{\la}{\lambda}

\newcommand{\su}{\; + \;}
\newcommand{\re}{\; - \;}
\newcommand{\ig}{\; = \;}

\newcommand{\ab}[1]{|#1|}

\title{The Discrete and Continuous Markus--Yamabe Stability Conjectures}

\author{\'Alvaro Casta\~neda and V\'ictor Gu\'i\~nez
\thanks{The authors were supported in part by FONDECYT Grant
\#1080172, by CONICYT Grant PBCT ADI 17, and by MATH-AMSUD DySET.
The first author was also supported by FONDECYT Postdoctoral Grant
\#3100082 and Mecesup PUC-0711.}}

\date{}

\begin{document}

\maketitle

\begin{abstract}

We study  the discrete and continuous versions of the Markus--Yamabe
Conjecture for polynomial vector fields in $ \R^n $ (especially when
$ n = 3 $) of the form $ X = \lambda \, I + H $ where $ \lambda $ is
a real number, I  the identity map, and H  a map with nilpotent
Jacobian matrix $ JH .$ We consider  the case where the rows of $ J
H $ are linearly dependent  over $ \R $ and that where they are
linearly independent over $ \R $. In the former, we find
non--linearly triangularizable vector fields $ X $
 for which the origin is a global attractor for both the continuous and the discrete dynamical systems generated by $ X $.
 In the independent continuous case, we present a family of vector fields which  have orbits escaping to infinity.
 In the independent discrete case, we present a large family of vector fields which have a periodic point of period 3.
\end{abstract}

\section{Introduction}

Let $ F: \R^n \to \R^n $ be a $C^1-$vector field with $F(p) = 0.$
Consider the differential system
\begin{equation} \label{1}
\dot{x} = F(x) \, .
\end{equation}
We let  $\phi(t,x)$ denote the solution of (\ref{1}) with initial
condition $\phi(0,x) = x $.  We say that $p$ is a {\em global
attractor} of the differential system (\ref{1}) if for each $ x \in
\R^n $, we have that $\phi(t,x)$ is defined for all $t > 0$ and
tends to $p$ as $t$ tends to infinity.

In \cite{MY}, L. Markus and H. Yamabe establish their well known
global stability conjecture.

\bigskip

\noindent {\bf{The Markus--Yamabe Conjecture (MYC).}} Let $ F: \R^n
\to \R^n$ be a $C^1-$ vector field with $ F(0) = 0 $. If for any $ x
\in \R^n$ all  the eigenvalues of  the Jacobian of $ F $ at $ x $
have negative real part, then the origin is a global attractor of
the differential system~(\ref{1}).

The corresponding version of the MYC for discrete dynamical systems is as follows. Let $ F: \R^n \to \R^n $ be a $C^1-$vector field. Consider the sequence
\begin{equation} \label{2}
x^{(m+1)} = F(x^{(m)}) \, , \quad x^{(0)} \in \R^n \, .
\end{equation}
Consider also the dynamics of the iterations of $ F .$ Let $ p $ be
a fixed point of $ F $, that is, $ F(p) = p $. We say that $ p $ is
a {\em global attractor} of the discrete dynamical system~(\ref{2})
if the sequence $ x^{(m)} $  tends to $ p $ as $ m $ tends to
infinity, for any $ x^{(0)} \in \R^n $.

\bigskip

\noindent {\bf{The Discrete Markus--Yamabe Conjecture (DMYC).}} Let $ F: \R^n
\to \R^n$ be a $C^1-$vector field with $ F(0) = 0 $. If for any $ x
\in \R^n$ all  the eigenvalues of  the Jacobian of $ F $ at $ x $
have modulus less than one, then the origin is a global attractor of the discrete dynamical system~(\ref{2}) generated by $ F $.

\bigskip

It is known that the MYC (resp. the DMYC) is true when $ n \leq 2 $
(resp. $ n = 1 $) and false when $ n \geq 3 $ (resp. $ n \geq 2 $),
though  both conjectures are true for triangular vector fields in
any dimension. For these vector fields, L. Markus and H. Yamabe
prove the continuous case  in \cite{MY}, and A. Cima et al. prove
the discrete case in \cite{CGM2}. For polynomial vector fields, the
DMYC is also true when $ n = 2 $ (see \cite{CGM2}), though both
conjectures are false  when $ n \geq 3 $. For an example of a pair
of polynomial vector fields, of which one satisfies the MYC
hypotheses and the other  the DMYC hypotheses, having both vector
fields orbits that escape to infinity, see \cite{CEGMH}. Further in
~\cite{CGM1}, A. Cima et al. obtain a family of polynomial
counterexamples containing the preceding pair.

\medskip
In this paper we  study both conjectures in the case of a special
family of polynomial vector fields in $ \R^n $, focusing on  $ n = 3
$. Given a real number $ \lambda $  and a positive integer $ n $, we
denote the set consisting of the polynomial vector fields in $ \R^n
$ of the form $ F = \lambda I + H $, where $ I $ is the identity map
and $ H $ has nilpotent Jacobian matrix at every point, by $
\mathcal{N}(\lambda,n) $.  Note that for this class of vector
fields, the Jacobian matrix at each $ x \in \R^n $ has all its
eigenvalues  equal to $ \lambda $. Therefore, a vector field $ F =
\lambda I + H $ in $ \mathcal{N}(\lambda,n) $ satisfies the MYC
(resp. the DMYC) hypotheses if and only if $ \lambda < 0 $ (resp. $
\ab{\lambda} \, < \, 1 $). The counterexamples of \cite{CGM1} are,
basically, vector fields $ X = \lambda I + H $ in $
\mathcal{N}(\lambda,3) $ where $ H $ is a quasi--homogeneous vector
field of degree one.  In \cite{GC}, we give examples of vector
fields in $ \mathcal{N}(\lambda,3) $ which are linearly
trianguralizable (that is, triangular after a linear change of
coordinates). For these vector fields, the MYC (resp. the DMYC) is
true when $ \lambda < 0 $ (resp.  $ \ab{\lambda} < 1 $). Further,
the paper contains a family of counterexamples to the MYC which
generalizes that of Cima--Gasull--Ma\~nosas.

Polynomial vector fields $ H $ defined on $ \R^n $ and on $ \C^n $
with nilpotent Jacobian matrix at every point have been extensively
studied from the algebraic geometry viewpoint (see for example \cite{vE}). In this paper we make use of some aspects of this theory.

 The examples and counterexamples $ X = \lambda \, I + H \in \mathcal{N}(\lambda,n) $  of above have one
common characteristic, namely the rows of $ JH $ are linearly dependent over $ \R $. Thus we are led  to introducing the sets $ \mathcal{N}_{ld}(\lambda,n) $ and $ \mathcal{N}_{li}(\lambda,n) $. The first is the set consisting of the polynomial vector
fields $ X = \lambda \, I + (H_1,\dots,H_n) $ in $ \mathcal{N}(\lambda,n) $
 such that $ \{H_1,\dots,H_n\} $ is linearly dependent over $ \R $. The second set is
 $ \mathcal{N}_{li}(\lambda,n) = \mathcal{N}(\lambda,n) - \mathcal{N}_{ld}(\lambda,n) $. Section 2 studies the linearly dependent case, especially when the dimension is three. We give a normal form for the vector fields of $ \mathcal{N}_{ld}(\lambda,3) $ (see Proposition~\ref{prop2.2}) and characterize  those  elements
which are linearly triangularizable (see Theorem~\ref{litri}). The normal form depends on a polynomial $ f(t) $ with coefficients in $ \R[z] $. In the case $ f(t) $ is a polynomial of degree one, we show that the corresponding vector fields satisfy  both conjectures (see Theorems~\ref{CY1} and \ref{DY1}). We thus  obtain  a family of non--linearly triangularizable vector fields in $ \mathcal{N}_{ld}(\lambda,3) $   for which both conjectures are true. To our knowledge, there are no examples as the preceding one in the literature. In the case the degree of $ f(t) $ is greater than one, we  give a new family of counterexamples to both the MYC and the DMYC (see Proposition~\ref{Y2}). The foregoing considerations lead us to  raising the question, Do there exist vector fields in $ \mathcal{N}_{ld}(\lambda,3) $ with the degree of $ f(t) $ greater than one for which  the MY Conjecture, or the DMY Conjecture, or both, are true? The section concludes showing that,  for a vector field $ X \in  \mathcal{N}_{ld}(\lambda,3) $, in order for the origin not to be a global attractor the vector field must have at least one orbit which escapes to infinity (see Theorem~\ref{period}).

In Section 3 we deal with the linearly independent case. We state
the Dependence Problem and the Generalized Dependence Problem
introduced by A. van den Essen in \cite[Chapter 7]{vE},  among
others, and we obtain a family of examples $ F_{n,r} = \lambda \, I
+ H_{n,r} $ in $ \mathcal{N}_{li}(\lambda,n) $ for any dimension $ n
\geq 3 $, with $ \textrm{rk} \; J H_{n,r} = r \geq 2 $. When $ n
\geq 3 $ and $ r = 2 $, we show that both conjectures are false for
these vector fields (see Theorem~\ref{liorsc}). Subsequently, we
consider vector fields $ X = \lambda \, I + H \in
\mathcal{N}_{li}(\lambda,3) $, where $ H(x,y,z) = (u(x,y,z),
v(x,y,z), h(u(x,y,z), v(x,y,z))) $. For a characterization of a
large class of these vector fields $ H ,$ see \cite{ChE}. The
characterization depends on a polynomial map $ g(t) $. In the case $
g(t) $ is a polynomial of degree less than or equal to two, we show
that the vector field $ X = \lambda \, I + H $, with $ \lambda < 0
,$ has orbits that escape to infinity (see Theorem~\ref{teo1li}). On
the other hand, in the discrete case, for $ |\lambda| < 1 $, these
maps have a periodic point of period three (see
Theorem~\ref{teo3.9}). Therefore, the DMYC is false for this class
of maps. This scenario leads to posing the following question. Do
there exist vector fields in $ \mathcal{N}_{li}(\lambda,3) $ for
which the MY Conjecture, or the DMY Conjecture, or both, are true?

\section{The linearly dependent case}

Given a linear isomorphism  $ T : \R^n \to \R^n $ and a vector field $ F = \lambda \, I + H \in \mathcal{N}(\lambda,n) $, we have
 \begin{equation} \label{camco} T_* F  = T \circ F \circ T^{-1} = \lambda \, I \su  T \circ H \circ T^{-1} \in \mathcal{N}(\lambda,n) \,
 \end{equation}
which implies that the set $ \mathcal{N}(\lambda,n) $ is invariant
by linear changes of coordinates as both a continuous dynamical
system  and  a discrete dynamical system. Moreover, the vector field
obtained after  a linear change of coordinates  is the same in both
the continuous and the discrete cases.

\medskip

Let $ \mathcal{N}_{ld}(\lambda,n) $ be the set consisting of the polynomial vector
fields $ X = \lambda \, I + (H_1,\dots,H_n) $ in $ \mathcal{N}(\lambda,n) $
 such that $ \{H_1,\dots,H_n\} $ is linearly dependent over $ \R $. Let $ \mathcal{N}_{li}(\lambda,n) = \NN(\lambda, n) - \mathcal{N}_{ld}(\lambda,n)$. Note that the examples and counterexamples of \cite{GC} belong to the set $ \mathcal{N}_{ld}(\lambda,3) $.
The following gives  properties of these sets.
\begin{prop}{\label{prop2.1}}
\begin{enumerate}
\item[1)] The sets $ \mathcal{N}_{ld}(\lambda,n) $ and $ \mathcal{N}_{li}(\lambda,n) $ are invariant by linear changes of coordinates.
\item[2)] Let $ X = \lambda \, I + (H_1,\dots,H_n) \in \mathcal{N}(\lambda,n) $
be such that $ X(0) = 0 $. Then $ X \in  \mathcal{N}_{ld}(\lambda,n) $ if and only if the rows  of the Jacobian
matrix \linebreak $ J(H_1,\dots,H_n) $  are linearly dependent over $ \R $.
\item[3)] Let $ X \in  \mathcal{N}_{ld}(\lambda,n) $. Then there exists a $ T \in Gl_n(\R) $ such that $ T_* X \ig \lambda \, I \su (H_1,\dots,H_{n-1},0) $. Moreover, for any $ a \in \R $, the vector field $ X_a : \R^{n-1} \to \R^{n-1} $ defined by $ X_a(x_1,\dots,x_{n-1}) = \lambda \, (x_1,\dots,x_{n-1}) \su (H_1,\dots,H_{n-1})((x_1,\dots,x_{n-1},a) $ belongs to $ \mathcal{N}(\lambda,n-1) $.
\end{enumerate}
\end{prop}
\begin{proof} Assertion 1) is clear. Assertion 2) follows from \cite[Exercise 7.1.1]{vE}.  Concerning  assertion 3), let $ X = \lambda \, I + (G_1,\dots,G_n) \in \mathcal{N}_{ld}(\lambda,n) $, and let
$ (\alpha_1,\dots,\alpha_n) \in \R^n - \{(0,\dots,0\} $ be such that $ \alpha_1 \, G_1 + \cdots + \alpha_n \, G_n \equiv 0 $. Without loss of
generality, we may suppose  $ \alpha_n \neq 0 $. Consider the
linear change of coordinates
$$ (x_1,\dots,x_n) \ig T(u_1,\dots,u_n) = (u_1,\dots,u_{n-1},\alpha_1 \, u_1 + \cdots + \alpha_n \,
u_n) \, . $$ Then  $ T_* X \ig \lambda \, I \su (H_1,\dots,H_{n-1},0) $ and the assertion follows easily.
\end{proof}
Our next result gives a normal form for  the vector fields in  $ \mathcal{N}_{ld}(\lambda,3) $. For a  proof, see  for example \cite[Corollary 1.1]{ChE}.

\begin{prop}{\label{prop2.2}} Let $ X = \lambda \, I + (S,U,V) \in \mathcal{N}_{ld}(\lambda,3)$.
Then there exists a $ T \in Gl_3(\R) $ such that $ T_* X \ig \lambda \,
I \su (P,Q,0) $ where
\begin{eqnarray} \label{normal}
P(x,y,z) & = &  -b(z) \, f(a(z) \, x + b(z) \, y) \su c(z) \quad \textrm{and} \nonumber \\
Q(x,y,z) & = &   a(z) \, f(a(z) \, x + b(z) \, y) \su d(z)
\,
\end{eqnarray}
with $ a, b, c, d \in \R[z] $ and $ f \in \R[z][t] $.
\end{prop}

\begin{rem} \label{reduc} In the normal form (\ref{normal}) we may assume $ f(0) = 0 $ by modifying the polynomials $ c(z) $ and $ d(z) $ if necessary.
\end{rem}

An interesting question about the vector fields satisfying the hypotheses of the MYC or the DMYC concerns the injectivity.

\begin{prop}{\label{prop2.3}}
Any $ X  \in \mathcal{N}_{ld}(\lambda,3) $ is
injective.
\end{prop}
\begin{proof} The Proposition results from the normal form (\ref{normal}).
\end{proof}

\medskip

Consider the following sets. Let $ \mathcal{N}_{CY}(\lambda,n) $ (resp. $ \mathcal{N}_{DY}(\lambda,n) $) be the subset of $ \mathcal{N}(\lambda,n) $ consisting of the
polynomial vector fields $ X $ such that
the origin is a global attractor for  the differentiable system $ \dot{x} = X(x) $ (resp. the discrete dynamical system generated by $ X $).

Recall that a vector field $ F :
\R^n \to \R^n $ is {\em triangular} if it has the form
$$ F(x_1,x_2, \cdots,x_n) \ig (F_1(x_1),F_2(x_1,x_2),\cdots,F_n(x_1,x_2,
\cdots,x_n)) \, . $$
Let $ \mathcal{N}_{CT}(\lambda,n) $ (resp. $ \mathcal{N}_{DT}(\lambda,n) $) be the subset of $ \mathcal{N}(\lambda,n) $ consisting of the
polynomial vector fields $ F $ for which there exists a diffeomorphism \linebreak $ G : \R^n \to \R^n $ such that $ G_* F $ (resp. $ G \circ F \circ G^{-1} $) is triangular. The elements of $ \mathcal{N}_{CT}(\lambda,n) $ (resp. $ \mathcal{N}_{DT}(\lambda,n) $) are called {\em triangularizable}. An element $ F $ in $ \mathcal{N}_{CT}(\lambda,n) $ (resp. $ \mathcal{N}_{DT}(\lambda,n) $) is said to be {\em linearly triangularizable} if there exists
a linear change of coordinates which makes $ F $ triangular. We denote the set consisting of the vector fields in $ \mathcal{N}(\lambda,n) $ which are linearly triangularizable  by
$ \mathcal{N}_{LT}(\lambda,n) $.
Since the MYC is true for $ C^1-$vector fields in dimension two and the DMYC is  true for polynomial vector fields also in dimension two, and since  both conjectures are true for triangular vector fields in any dimension, we have the following.

\begin{teo} Let $ \lambda \in \R $, and let $ n \geq 2 $ be an integer. \\Then:
\begin{enumerate}
\item[a)] If  $ \lambda < 0 $ (resp. $ \ab{\lambda} < 1 $), then
$ \mathcal{N}(\lambda,2) \ig  \mathcal{N}_{CY}(\lambda,2) $ (resp. $ \mathcal{N}(\lambda,2) \ig  \mathcal{N}_{DY}(\lambda,2) $).
\item[b)] If  $ \, \lambda < 0 $, then
$  \mathcal{N}_{LT}(\lambda,n) \subset \mathcal{N}_{CT}(\lambda,n)  \subset \mathcal{N}_{CY}(\lambda,n) \cap \mathcal{N}_{ld}(\lambda,n) \, . $
\item[c)]If $ \, 0 < \ab{\lambda} < 1 $, then
$  \mathcal{N}_{LT}(\lambda,n) \subset \mathcal{N}_{DT}(\lambda,n)  \subset \mathcal{N}_{DY}(\lambda,n) \cap \mathcal{N}_{ld}(\lambda,n) \, . $
\end{enumerate}
\end{teo}
\begin{proof} For a proof of assertion a) in the continuous case, see for example C. Gutierrez~\cite{Gu}; in the discrete case, see \cite[Theorem B]{CGM2}. Assertion~b) follows from \cite[Theorem 4]{MY}. Assertion c) follows from \cite[Theorem~A]{CGM2}.
\end{proof}

Using the preceding notation, the main result of \cite[Section 2]{GC} and its corollaries
can be rewritten as follows.
\begin{teo} \label{principal} Let $ \lambda \in \R $, and let $ m \geq 1 $ be
an integer. Assume $ X_k = \lambda I + H_1 + \cdots + H_k  \in
\mathcal{N}_{ld}(\lambda,3)$ for $ 1 \leq k \leq m $, where $ H_i $ is
a homogeneous polynomial of degree $ i $, with $ i = 1,\dots,m $.
Then $ X_k \in  \mathcal{N}_{LT}(\lambda,3)$ and
$$ X_m(x,y,z) \ig \lambda \, (x,y,z) \su (0,a_1 \, x + \cdots a_m \, x^m,
r_1(x,y) + \cdots + r_m(x,y)) \,  $$
where $ r_i(x,y) $ is a
homogeneous polynomial of degree $ i $, with $ 1 \leq i \leq m $,  up to a linear change of coordinates.
\end{teo}

\begin{cor} \label{kmasm} Let $ \lambda \in \R $, and let $ k, m $
be integers such that \linebreak $ 1 \leq k < m $. Any polynomial vector field $ X
= \lambda I + H_k + H_m \in \mathcal{N}(\lambda,3) $, with  $ H_k $
and $ H_m $  homogeneous polynomials of degree k and m,
respectively, belongs to $ \mathcal{N}_{LT}(\lambda,3) $.
\end{cor}

\begin{cor} \label{2mas3} Let $ \lambda \in \R $. Any polynomial vector field
$ X = \lambda I + H_2 + H_3 \in \mathcal{N}(\lambda,3) $, with  $
H_2 $ and $ H_3 $  homogeneous polynomials of degree 2 and 3, respectively, belongs
to $ \mathcal{N}_{LT}(\lambda,3) $.
\end{cor}

\medskip

Our next result establishes conditions under which vector
fields of the form $X = \lambda I + (P, Q, 0) \in
\mathcal{N}_{ld}(\lambda,3)$, with $(P,Q)$ as in  Proposition \ref{prop2.2},
are linearly triangularizable.

\begin{teo} \label{litri} Let $ X = \lambda \, I + H \in \mathcal{N}_{ld}(\lambda,3) $ where
$$ H(x,y,z) \ig  f(a(z) \, x + b(z) \, y) \, (-b(z),a(z),0) \su (c(z),d(z),0)
\,  $$ with $ \lambda \in \R $, $ a,b,c,d \in \R[z] $, $ f \in \R[z][t] $,
and $ X(0) = 0 $. Then
 $ X \in
\mathcal{N}_{LT}(\lambda,3) $ if and only if either $ f $ is constant or $
\{a,b\} $ are linearly dependent over $ \R $.
\end{teo}
\begin{proof} When $ f $ depends only on $ z $, the result is clear. In what follows, we will assume that
the degree of $ f $ with respect to $ t $ is greater than zero. If $ \{a,b\} $ are linearly
dependent over $ \R $, then there exists $ (\alpha,\beta) \in \R^2 -
\{(0,0)\} $ such that $ \alpha \, a(z) \su \beta \, b(z) \ig 0 $,
for all $ z \in \R $. Assume $ \beta \neq 0 $. Then $ b(z) = \delta
\, a(z) $, with $ \delta = -\frac{\alpha}{\beta} $. Consider the
linear isomorphism $ T(x,y,z) = (z, x + \delta \, y, y) $. Then
$$ T_*(X)(u,v,w) \ig  \lambda \, (u,v,w) \su (0,c(u) + \delta \, d(u),
a(u) \, f(a(u) \, v) + d(u)) \, , $$ which is triangular.

\medskip

Now suppose that there exists a linear isomorphism $ M $ such that
$$ M_*(X)(u,v,w) \ig  \lambda \, (u,v,w) \su (A(v,w), B(w), 0) \, . $$
Assume that $ [M] \ig (m_{ij})_{1 \leq i,j \leq 3} $ is the
matrix of $ M $ with respect to the canonical basis of $ \R^3 $. We
have
$$ m_{31} \, [-f(t) \, b(z) + c(z)] \su m_{32} \, [f(t) \, a(z) + d(z)] \equiv 0 \,  $$
where $ t \ig a(z) \, x + b(z) \, y $. Then
$$ m_{31} \, [-f(0) \, b(z) + c(z)] \su m_{32} \, [f(0) \, a(z) + d(z)] \equiv 0 \,  $$
and, therefore,
$$ (f(t) - f(0)) \, [-m_{31} \, b(z)  \su  m_{32} \, a(z)] \equiv 0 \, . $$
If $ (m_{31}, m_{3,2}) \neq (0,0) $,  the proof is complete. If $
(m_{31}, m_{3,2}) \ig (0,0) $, we may assume that $ m_{33} = 1 $ and
 $ \det[M] = 1 $. Thus the matrix of $ M^{-1} $ with
respect to the canonical basis of $ \R^3 $ is
$$ [M^{-1}] \ig \left(
                  \begin{array}{ccc}
                    m_{22} & -m_{12} & \tilde m_{13} \\
                    -m_{21} & m_{11} & \tilde m_{23} \\
                    0 & 0 & 1 \\
                  \end{array}
                \right) $$
with $ \tilde m_{13} \ig -m_{13} \, m_{22} + m_{12} \, m_{23} $ and
$ \tilde m_{23} \ig m_{13} \, m_{21} - m_{11} \, m_{23} $. Then
$$ t = a(w) \, [m_{22} \, u - m_{12} \, v + \tilde m_{13} \, w] \su
b(w) \, [-m_{21} \, u + m_{11} \, v + \tilde m_{23} \, w] \,  $$
and
$$ B(w) \ig m_{21} \, [-f(t) \, b(w) + c(w)] \su m_{22} \, [f(t) \, a(w) + d(w)] \, . $$
Differentiating the preceding expression with respect to $ u $ we obtain
$$ 0 \ig f'(t) \, [m_{22} \, a(w) - m_{21} \, b(w)]^2  \,  $$
and so $ \{a,b\} $ are linearly dependent over $ \R $, which completes  the
proof.
\end{proof}

The next two results assert  that, in the  linearly dependent case,  the origin is always a global attractor when the degree  of the polynomial $ f(t) $ is one. These results give examples of vector fields in $ \mathcal{N}_{CY}(\lambda,3) $ and in $ \mathcal{N}_{DY}(\lambda,3) $ which are not linearly triangularizable. By Remark~\ref{reduc}, it suffices to consider the case $ f(t) = g(z) \, t $.

\begin{teo} \label{CY1} Let $ X = \lambda \, I + H \in \mathcal{N}_{ld}(\lambda,3) $ where
$$ H(x,y,z) \ig  g(z) \, (a(z) \, x + b(z) \, y) \, (-b(z),a(z),0) \su (c(z),d(z),0)
\,  $$ with $ \lambda < 0 $, $ a,b,c,d,g \in \R[z] $ and $ X(0) = 0
$. Then
 $ X \in
\mathcal{N}_{CY}(\lambda,3) $.
\end{teo}
\begin{proof} Note that $ (x(t),y(t),z(t)) $ is a solution of the differential system $ \dot{x} = X(x) $ if and only if $ z(t) = z_0 \, e^{\lambda t} $ and $ (x(t), y(t)) $ is a solution of the linear system
$$ \left(
     \begin{array}{c}
       \dot{x} \\
       \dot{y} \\
     \end{array}
   \right)  \ig \left(
                  \begin{array}{cc}
                    \lambda - A(t)B(t)G(t) & -B(t)^2 G(t) \\
                    A(t)^2 G(t) &  \lambda + A(t)B(t)G(t)\\
                  \end{array}
                \right) \; \left(
     \begin{array}{c}
       x \\
       y \\
     \end{array}
   \right)  \su \left(\begin{array}{c}
       C(t) \\
       D(t) \\
     \end{array}
   \right)\,  $$
where $ (A,B,C,D,G)(t) \ig (a,b,c,d,g)(z_0 \, e^{\lambda t}) $.
Since the origin is a local attractor, there is a basis of solutions of the linear system consisting of solutions which tend to the origin as $ t $ tends to $ + \infty $. Therefore, the origin is a global attractor.
\end{proof}

\begin{teo} \label{DY1} Let $ F = \lambda \, I + H \in \mathcal{N}_{ld}(\lambda,3) $ where
$$ H(x,y,z) \ig  g(z) \, (a(z) \, x + b(z) \, y) \, (-b(z),a(z),0) \su (c(z),d(z),0)
\,  $$ with $ 0 < \lambda < 1 $, $ a,b,c,d,g \in \R[z] $, and $ X(0) = 0
$. Then
 $ F \in
\mathcal{N}_{DY}(\lambda,3) $.
\end{teo}
\begin{proof} Without loss of generality, we may assume that  $ c(z) \equiv  d(z) \equiv 0 $. In fact, the polynomials $ c(z) $ and $  d(z) $ may be eliminated by applying the coordinate change  $ T(u,v,w) = (u + m(w), v + n(w), w) $ where

$$ \left(
     \begin{array}{c}
       m(w) \\
       n(w) \\
     \end{array}
   \right) = \frac{-1}{(1 - \lambda)^2} \, \left[(1 - \lambda) \, I \su g(w) \, M(w)\right]  \left(
                                                           \begin{array}{c}
                                                             c(w) \\
                                                             d(w) \\
                                                           \end{array}
                                                         \right)  $$
with
$$ M(w) \ig \left(
                                                \begin{array}{cc}
                                                  -a(w)b(w) & -b(w)^2  \\
                                                  a(w)^2  &  a(w)b(w)  \\
                                                \end{array}
                                              \right) \, . $$
So we assume
   $$ H(x,y,z) \ig  g(z) \, (a(z) \, x + b(z) \, y) \, (-b(z),a(z),0) \, . $$
Therefore,

$$ F(x,y,z) \ig
(A(z)\left(
                     \begin{array}{c}
                       x \\
                       y \\
                     \end{array}
                   \right), \lambda \, z) \,   $$ where
$$ A(z) \ig \left(
              \begin{array}{cc}
                \lambda - a(z)b(z)g(z) & -b(z)^2 g(z) \\
                a(z)^2 g(z) & \lambda  + a(z)b(z)g(z) \\
              \end{array}
            \right) \, . $$
Thus it  suffices to prove that, for any $ (x,y) \in \R^2 $, we have
$$ \lim_{n \to \infty} A(\lambda^n z)A((\lambda^{n-1} z)\dots A(\lambda z) A(z) \left(
                     \begin{array}{c}
                       x \\
                       y \\
                     \end{array}
                   \right) \ig \left(
                     \begin{array}{c}
                       0 \\
                       0 \\
                     \end{array}
                   \right)  \, . $$

\medskip

Let $ \mathcal{M}_2 $ be the normal vector space of the $ 2 \times 2 $ real matrices $ A = (a_{ij}) $ endowed with the norm $ \norm{A} = 2 \, \max \, \ab{a_{ij}} $. Considering $ \R^2 $ endowed with  the norm $ \norm{(x,y)} = \max\{\ab{x},\ab{y}\} $, we have
$$ \norm{A((x,y))} \leq \norm{A} \, \norm{(x,y)} \quad \textrm{and} \quad \norm{A B} \leq \norm{A} \norm{B} \, . $$
A simple computation yields
$$ A(\lambda^n z)A((\lambda^{n-1} z)\dots A(\lambda z) A(z) = $$ $$ \left(
                                                                \begin{array}{cc}
                                                                  \lambda^{n} - n a_0 b_0 g_o \lambda^{n-1} + r_{11}(z) &  -n b_0^2 g_0 \lambda^{n-1} + r_{12}(z)\\
                                                                   n a_0^2 g_0 \lambda^{n-1} + r_{21}(z) & \lambda^{n} + n a_0 b_0 g_0 \lambda^{n-1} +  r_{22}(z) \\
                                                                \end{array}
                                                              \right) $$

\medskip

\noindent where $ r_{ij}(0) = 0 $ and $ (a_0,b_0,g_0) = (a,b,g)(0) $.

\medskip

\noindent Fix $ N \in \N $ so that $ 2 \, N \lambda^{N-1} \, \max\{a_0^2 g_0, b_0^2 g_0, |a_0 b_0 g_0| \} < 1 $. Let
$ B(z) \ig A(\lambda^{N-1} z)A((\lambda^{N-2} z)\dots A(\lambda z) A(z) $. Consider $ 0 < | z|< z_0 $ such that $  \norm{B(z)} \leq  K < 1 $. Then, for $ n = k N - 1 $, we have
\begin{eqnarray*} \norm{A(\lambda^n z)\dots A(z)\left(
                     \begin{array}{c}
                       x \\
                       y \\
                     \end{array}
                   \right)} & = & \norm{B(\lambda^{(k-1)N}z) \dots B(z)\left(
                     \begin{array}{c}
                       x \\
                       y \\
                     \end{array}
                   \right)} \\
                  & \leq & K^k \, \norm{(x,y)} \to 0 \quad  \textrm{if} \quad k \to \infty \,
                   \end{eqnarray*}
which completes the proof.
\end{proof}

When the degree  of the polynomial $ f(t) $ is greater than one, there are examples of polynomial vector fields in  $ \mathcal{N}_{ld}(\lambda,3) $ having orbits that escape to infinity.

\medskip

\begin{prop} \label{Y2} Let $ F = \lambda \, I + H \in \mathcal{N}_{ld}(\lambda,3) $ where
$$ H(x,y,z) \ig   z^{k-1} \, (x + y \, z)^{k + 1} \, (-z,1,0)
\, . $$
Then:
\begin{enumerate}
\item[a)] If $ \, \lambda < 0 $ and $ k $ odd, then
 $ F \not\in
\mathcal{N}_{CY}(\lambda,3) $.
\item[b)] If $ \, 0 < \ab{\lambda} < 1 $, then $ F \not\in
\mathcal{N}_{DY}(\lambda,3) $.
\end{enumerate}
\end{prop}
\begin{proof} a) By applying the coordinate change
\begin{equation} \label{camco}
(u,v,w) = T(x,y,z) = (z(x + yz), \lambda y z^2, z) \,
\end{equation}
outside of $ z = 0 $, we obtain
 \begin{equation} \label{system1}
 T_*(F)(u,v,w) \ig (2 \lambda \, u + v, 3 \lambda \, v + \lambda \, u^{k+1}, \lambda \, w) \, .
 \end{equation}
 The planar system
 \begin{equation} \label{system2}
 (\dot{u}, \dot{v}) \ig (2 \lambda \, u + v, 3 \lambda \, v + \lambda \, u^{k+1}) \,
 \end{equation}
has two singular points, namely the origin and the point $ u_0(1,-2\lambda) $ where $ u_0 =
\sqrt[k]{6 \lambda} \,$. For $ z_0 \neq 0 $,  the latter singular point generates the following orbit of the original system
$$ \gamma(t) = (\frac{3 \, \sqrt[k]{6 \lambda}}{z_0} \, e^{-\lambda t}, -\frac{2 \, \sqrt[k]{6 \lambda}}{z_0^2} \, e^{-2 \lambda t}, z_0 \, e^{\lambda t}) $$
which escapes to infinity as $ t $ tends to $ +\infty $.

\noindent b) Again, by applying the coordinate change (\ref{camco})
outside of $ z = 0 $, we obtain
 \begin{equation} \label{system2}
 T \circ F \circ T^{-1}(u,v,w) \ig \lambda \, (\lambda \, u + (\lambda - 1)(v + u^{k + 1}), \lambda^2 \, (v + u^{k+1}), w) \, .
 \end{equation}
 The planar map
 $$ (u,v) \to \lambda \, (\lambda \, u + (\lambda - 1)(v + u^{k + 1}), \lambda^2 \, (v + u^{k+1})) $$
has two fixed points, namely the origin and the point $ u_0(1, -\lambda^2(1 + \lambda)) $ where $ u_0 = \sqrt[k]{\frac{(1 + \lambda)(\lambda^3 - 1)}{\lambda}} \, \cdot $ For $ z_0 \neq 0 $,  the latter fixed point generates the orbit of the original map
 $$ (x_n,y_n,z_n) \ig \left(\frac{(1 + \lambda + \lambda^2)u_0}{\lambda^n z_0} ,-\frac{(1 + \lambda)u_0}{\lambda^{2n-1} z_0^2} , \lambda^n z_0\right) \, $$
 which escapes to infinity as $ n $ tends to $ +\infty $.
\end{proof}

\begin{rem}
 Assertion a) of Proposition~\ref{Y2} is also true if we
consider $ n \geq 2 $, $ a(z) = 1 $, $ b(z) = z $, and $ f(t) =
z^{n-2} \, (A_1(z) \, t + \dots + A_n(z) \, t^n) $, with $ n $
either even and $ A_n(0) \neq 0 $ or $ n $ odd and $ A_n(0) < 0 $.
\end{rem}

\medskip

\noindent Thus we are led to posing the following.

\medskip

\noindent {\bf Question 1.} Do there exist vector fields in $ \mathcal{N}_{ld}(\lambda,3) $, with the degree of $ f(t) $ greater than one, for which either the MY Conjecture, or the DMY Conjecture, or both, are true?

\medskip

Our next result shows that, for a vector field $ X \in  \mathcal{N}_{ld}(\lambda,3) $,
in order for the origin not to be a global attractor the vector field must have at least one orbit which escapes to infinity.

\begin{teo}\label{period} Let $ \lambda \in \R $, and let $ F = \lambda \, I + H \in \mathcal{N}_{ld}(\lambda,3) $.
If $ \lambda < 0 $ (resp. $ \ab{\lambda} < 1 $) and $ F \not\in
\mathcal{N}_{CY}(\lambda,3) $ (resp. $ F \not\in
\mathcal{N}_{DY}(\lambda,3) $, then the differential system $
\dot{x} = F(x) $ (resp.  discrete dynamical system generated by $ F
$) has orbits which escape to infinity.
\end{teo}
\begin{proof} We may assume that $ H = (P,Q,0) $ where
\begin{eqnarray*}
P(x,y,z) & = &  -b(z) \, f(a(z) \, x + b(z) \, y) \su c(z) \quad \textrm{and} \\
Q(x,y,z) & = &   a(z) \, f(a(z) \, x + b(z) \, y) \su d(z)
\end{eqnarray*}
with $ a, b, c, d \in \R[z] $ and $ f \in \R[z][t] $.

Consider the case $ \lambda < 0 $. Let  $ \gamma(t) =
(x(t),y(t),z(t)) $ be a solution of the system $ \dot{x} = F(x) $.
We denote the omega--limit set of $ \gamma $ by $ \omega(\gamma) $.
Since $ z(t) = z(0) \, e^{\lambda \, t} $, we have  $ \omega(\gamma)
\subset W_{\infty} $, where $ W_{\infty} $ is the extended plane $
\{z = 0 \}\cup \{\infty\} $. If the orbit $ \gamma(t) $ is bounded,
then $ \omega(\gamma) = \{0\} ,$ thus obtaining the theorem. The
proof  for the discrete dynamical system generated by $ F $ is
analogous.

\end{proof}

\bigskip

\section{The linearly independent case}
In this section we consider the linearly independent case. We begin
with some algebraic preliminaries extracted from \cite[Chapter
7]{vE} and \cite{ChE}. The study of the Jacobian Conjecture for
polynomial maps of the form $ I + H ,$ where I is the identity map
and H a homogeneous map of degree 3, with JH nilpotent, led various
authors to the following problem. Let $ \kappa $ be a field of
characteristic  zero.

\medskip

\noindent {\bf  Dependence Problem}. Let $d \in \N $, with $ d \geq 1$, and let
$ H = (H_1,\dots,H_n) : \kappa^n \to \kappa^n $ be a homogeneous
polynomial map of degree $d$ such that $ J H $ is nilpotent. Does it
follow that $H_1, \ldots, H_n$ are linearly dependent over $ \kappa
$?

\medskip

 The attempt to solve it by induction led to consider the more general problem:

\medskip

\noindent {\bf Generalized Dependence Problem}. Let $ H =
(H_1,\dots,H_n) : \kappa^n \to \kappa^n $ be a polynomial map such
that $ J H $ is nilpotent. Are the rows of $ J H $ linearly
dependent over $ \kappa $?

\medskip

The answer to this question turned out to be \lq \lq yes" if $n \leq
2$ and \lq \lq no" if $n \geq 3.$ More precisely, the following is
proved by van den Essen in \cite[Theorem 7.1.7]{vE}.

\begin{teo} \begin{enumerate}
\item[i)] If $ J H $ is nilpotent and $ \textrm{rk} \; J H \leq 1 $, then the rows
of $ J H $ are linearly dependent over $ \kappa $ (here rk is the
rank as an element of $ M_n(\kappa(X)) $.
\item[ii)] Let $ r \geq 2 $. Then, for any dimension $ n \geq r + 1 $, there exists
a polynomial map  $ H_{n,r} : \kappa^n \to \kappa^n $ such that $ J H_{n,r} $ is
nilpotent, $ \textrm{rk} \; J H_{n,r} = r $, and the rows of $ J H_{n,r} $ are
linearly independent over $ \kappa $.
\end{enumerate}
\end{teo}

As an example, let $ a \in \R[x_1] ,$ with $ \textrm{deg} \, a = r $
and $ f(x_1,x_2) = x_2 - a(x_1) $. Then $ H_{n,r} = (H_1,\dots,H_n)
,$ where
\begin{eqnarray*}
H_1(x_1,\dots,x_n) & = & f(x_1,x_2) \, , \\
H_i(x_1,\dots,x_n) & = & x_{i+1}  + \frac{(-1)^i}{(i-1)!} \,
a^{(i-1)}(x_1) \,
(f(x_1,x_2))^{i-1} \, , \; \textrm{if} \; 2 \leq i \leq r \, ,\\
H_{r+1}(x_1,\dots,x_n) & = & \frac{(-1)^{r + 1}}{r!} \, a^{(r)}(x_1)
\,
(f(x_1,x_2))^{r} \, , \; \textrm{and} \\
H_j(x_1,\dots,x_n) & = & (f(x_1,x_2))^{j-1} \, , \; \textrm{if} \; r
+ 1 < j \leq n \, ,
\end{eqnarray*}
is a polynomial map satisfying  assertion ii). (See
\cite[Proposition 7.1.9]{vE}). For $ r = 2 $ and $ n \geq 3 $, the
components of $ H_{n,2} $ are
\begin{eqnarray}\label{li}
H_1(x_1,\dots,x_n) & = & x_2 - a \, x_1 - b \, x_1^2 \, , \nonumber\\
H_2(x_1,\dots,x_n) & = & x_3 + (a + 2 \, b \, x_1) \, (x_2 - a \, x_1 - b \, x_1^2) \,
, \; \\
H_3(x_1,\dots,x_n) & = & -b \, (x_2 - a \, x_1 - b \, x_1^2)^2 \, , \nonumber
   \textrm{and for} \; j \geq 4  \\
H_j(x_1,\dots,x_n) & = &  (x_2 - a \, x_1 - b \, x_1^2)^{j-1} \,  \nonumber
 \, ,
\end{eqnarray}
with $ b \neq 0 $.

\bigskip

\begin{teo} \label{liorsc} Let $ F_{n,2} = \lambda \, I + (H_1,\dots,H_n) $,
with  $ H_i $ as in (\ref{li}).
\begin{enumerate}
\item[a)] If $ \lambda < 0 ,$ then the system $ \dot{x} = F_{n,2}(x) $ has orbits that escape to infinity.
\item[b)] If $ -1 < \lambda < 1 ,$ then the discrete dynamical system generated by $ F_{n,2} $ has a periodic orbit of period three.
\end{enumerate}
\end{teo}

\begin{proof} For assertion a), it suffices to prove the Theorem when $ n = 3 $.
Set $ X = F_{3,2} .$ Then by applying the coordinate change $
(u,v,w) = \phi(x_1,x_2,x_3) = b \, (x_1, x_3 - \lambda \, b \,
x_1^2, \lambda \, x_1 + x_2 - a \, x_1 - b \, x_1^2) $, we have $
\phi^*(X) = Y $ where
$$ Y(u,v,w) \ig (w,\lambda \, v - w^2, 2 \, \lambda \, w + v - \lambda^2 \, u)
\, . $$ We make the change of coordinates
$$ (s,q,p) \ig \frac{1}{v}(1,u,w) $$
to find those orbits of $Y$ which escape to infinity.  If $ Z $ is
the vector field $ Y $ in the new coordinates, then $ W \ig
(W_1,W_2,W_3) \ig s \, Z $ is defined by
$$ W(s,q,p) \ig (-s (\lambda \, s - p^2),
s (p - \lambda \, q) + q \, p^2, s (\lambda \, p + 1 - \lambda^2 \,
q) + p^3)
 \, . $$
For $ s \neq 0 $, the orbits of $ Z $ and  $ W $ are the same.
Moreover, for $ s > 0 $ (resp. $ s < 0 $), the orbits of $ Z $ and $
W $ have the same (resp. inverse) orientation. Over the plane $ s =
0 $,  the vector field $ W $ is radially repelling outside of a line
of singular points, namely the line $ p = 0 $. For $ s > 0 $, we
have $ W_1 > 0 $ and, therefore, there are no orbits there with $
\omega-$limit set contained at $ s = 0 $. For $ s < 0 $, we must
find orbits of $ W $ with $ \alpha-$limit set contained at $ s = 0
$. Indeed, consider the numbers $$ A = 2 \, \lambda \, , \quad s_0 =
\frac{1}{512 \lambda^3} \, ,  \quad p_0 = - \frac{1}{8 \, \lambda}
\, , \quad q_0 = \frac{11}{16 \, \lambda^2} $$ and the set
$$ P_A = \{(s,q,p) : As - p^2 \leq 0 \, , s_0 \leq s \leq 0 \, , 0 \leq q \leq q_0
\, , 0 \leq p \leq p_0 \} \, . $$ We find:
\begin{enumerate}
\item[1)] Over the set
$ P_A \cap \{(s,q,p) : A \, s - p^2 = 0\} $, the vector field $ W $
points outward from the set $ P_A $.  In fact, if $ (s,q,p) \in  P_A $
and $ A \, s - p^2 = 0 $, then
\begin{eqnarray*} A \, W_1 - 2 \, p \, W_3 & = & -\frac{p^3}{A} \,
[p \, (A + 3\lambda) \su
2 \,(1 - \lambda^2 \, q)] \\
& \geq &  -\frac{p^3}{A} \,[p_0 \, (A + 3\lambda) \su 2 \,(1 -
\lambda^2 \, q_0)] = 0 \, .
\end{eqnarray*}
\item[2)] Over the set $ P_A \cap \{(s,q,p_0) :  s < 0\} $, the vector field
$ W $ points outward from the set $ P_A $. In fact, if $ p = p_0 $, then
\begin{eqnarray*}
W_3 & = & s \, (\lambda \, p_0 + 1 - \lambda^2 \, q) + p_0^3 \\
& \geq & s \, (\lambda \, p_0 + 1) + p_0^3 \ig \frac{7}{8} \, s -
\frac{1}{8^3 \, \lambda^3} \\
& \geq & \frac{7}{8} \, s_0 - \frac{1}{8^3 \, \lambda^3} \ig
\frac{-1}{8^4 \, \lambda^3} \, > \, 0 \, .
\end{eqnarray*}
\item[3)] Over the set $ P_A \cap \{(s,0,p)\} $, the vector field
$ W $ points outward from the set $ P_A $. In fact, if $ q = 0 $, then
$$ W_2 \ig s \, p \, \leq \, 0 \, . $$
\item[4)] Over the set $ P_A \cap \{(s,q_0,p)\} $, the vector field
$ W $ points outward from the set $ P_A $. In fact, if $ q = q_0 $, then
$$ W_2 \ig s \, p  - q_0 \, (\lambda s - p^2) \ig (\lambda \, s - p^2) \,
[\frac{sp}{\lambda \, s - p^2} - q_0] \geq 0 \,  $$ because
$$ \lambda \, s - p^2 \, < \, A \, s - p^2 \leq A \, s_0 - \frac{p_0^2}{4} \ig 0 $$
and
$$ h(s,p)  \ig \frac{sp}{\lambda \, s - p^2} \leq h(s_0,p) \leq
h(s_0, \frac{p_0}{2}) \ig  \frac{1}{16 \lambda^2} \, < \, q_0
 \, . $$
\end{enumerate}
Thus, any orbit $ \gamma(t) $ of $ W $, with $ \gamma(0) $ an
interior point of $ P_A $, has $ \alpha-$limit set contained in the
line $ s = p = 0 $.  Clearly, any (of these)
 orbit corresponds to an orbit of our initial vector field $ X $ that escapes to infinity.
This completes the proof in this case.

Concerning assertion b), assume $n > 3.$ The proof of the case $n =
3$ is a particular case of Theorem~\ref{teo3.9}. Note that
$$ F_{n,2}(x_1,\dots,x_n) = (F_{3,2}(x_1,x_2,x_3), \lambda x_4  + f(x_1,x_2)^3, \dots, \lambda x_n + f(x_1,x_2)^{n-1}) $$
where $ f(x_1,x_2) = x_2 - a \, x_1 - b \, x_1^2 $. We have that the
third iterate of $ F_{n,2} $ is of the form
$$ F_{n,2}^3(x_1,\dots,x_n) = (F_{3,2}^3(x_1,x_2,x_3),\lambda^3 x_4 + g_4(x_1,x_2,x_3),\dots,\lambda^3 x_n + g_n(x_1,x_2,x_3)) \, . $$
The point $ (\overline{x_1},\dots,\overline{x_n}) $, where $
(\overline{x_1},\overline{x_2},\overline{x_3}) $ is a periodic point
of period three of $ F_{3,2} $ and $ \overline{x_j} = \frac{1}{1 -
\lambda^3} \, g_j(\overline{x_1},\overline{x_2},\overline{x_3}) $,
with $ 4 \leq j \leq n $, is a periodic point of period three of $
F_{n,2} ,$ which completes the proof.
\end{proof}

Note that $ H_{3,2} $ has the special form
\begin{equation} \label{implicit}
H_{3,2}(x,y,z) \ig (u(x,y), v(x,y,z), h(u(x,y))) \, .
\end{equation}

By applying a linear change of coordinates to a large class of
polynomial maps $ H = (H_1,H_2,H_3) $ of the form
\begin{equation} \label{implicitgeneral}
H(x,y,z) \ig  (u(x,y,z), v(x,y,z), h(u(x,y,z),v(x,y,z))) \,
\end{equation}
where $ JH $ is nilpotent such that $ H_1, H_2, H_3 $ are linearly
independent, we obtain

 \begin{equation} \label{implicitgeneralreduced}
 G(x,y,z) \ig  (g(t), v_1 \, z -(b_1 + 2v_1 \alpha \, x) \, g(t), \alpha \,  g(t)^2)
 \end{equation}
with $ t = y + b_1 \, x + v_1 \alpha \, x^2 $ and  $ v_1 \alpha \neq
0 $, and $ g \in \R[t] $ with $ g(0) = 0 $ and $ \textrm{deg}_t  \,
g(t) \geq 1 $. More specifically, we have the following.

\begin{teo}\label{cla}
Let $ H(x,y,z) = (u(x,y,z), v(x,y,z), h(u(x,y,z),v(x,y,z))) $.
Assume that $ H(0) = 0, h'(0) = 0 $, and the components of $ H $ are
linearly independent over $ \R $. Let $ A = \frac{\partial
v}{\partial x} \, \frac{\partial u}{\partial z} - \frac{\partial
u}{\partial x} \, \frac{\partial v}{\partial z} $ and $ B =
\frac{\partial v}{\partial y} \, \frac{\partial u}{\partial z} -
\frac{\partial u}{\partial y} \, \frac{\partial v}{\partial z} $. If
$JH$ is nilpotent and $ \textrm{deg}_z(uA) \neq \textrm{deg}_z(vB)
,$ then there exists a $ T \in GL_3(\R) $ such that $ THT^{-1} $ is
of the form (\ref{implicitgeneralreduced}).
\end{teo}

\noindent (See \cite{ChE}.)

\begin{rem} \label{rem1} \begin{enumerate}
\item[1)]  Under the condition $ \textrm{deg}_z(uA) \neq \textrm{deg}_z(vB) ,$ by Theorem~\ref{cla}, any vector field
$$ X = \lambda \, I + (u(x,y,z), v(x,y,z), h(u(x,y,z),v(x,y,z))) \in \mathcal{N}_{li}(\lambda,3) $$
 has the  form
 \begin{eqnarray} \label{linf}
 X(x,y,z) & = & \lambda \, (x,y,z) \su (0, v_1 \, z, 0) \su  \\
 & & g(t) \, (1, -(b_1 + 2v_1 \alpha \, x), \alpha \,  g(t))  \nonumber
 \end{eqnarray}
up to a linear change of coordinates, where $ t = y + b_1 \, x + v_1
\alpha \, x^2, \, v_1 \alpha \neq 0 $, and $ g \in \R[t] ,$ with $
g(0) = 0 $ and $ \textrm{deg}_t  \, g(t) \geq 1 .$
\item[2)] When $ n = 3 $, the vector field $\, F_{3,2} \,$ of Theorem~\ref{liorsc} - - up to a  linear change of coordinates - -
 has the form (\ref{linf}) with $ g(t) $ a polynomial of degree one.  Therefore, for this vector field both the MYC and the DMYC are false.
\end{enumerate}
\end{rem}

Consequently, we ask:

\medskip

\noindent {\bf Question 2.} Do there exist vector fields $ X  \in
\mathcal{N}_{li}(\lambda,3) $ of the form (\ref{linf}) for which the
MY and/or the DMY Conjecture is true?

\medskip

\subsection{The continuous case}

In the continuous case, our next result gives a negative answer to
Question 2  when the degree of $ g(t) $ is less than or equal to
two. First note that by applying the coordinate change
\begin{eqnarray*}
(u,v,w) & = & \phi(x,y,z) \\
& = & (\lambda \, x + g(t), t, v_1 \, z + \lambda \, v_1 \, \alpha \, x^2 )
\end{eqnarray*}
where $ t = y + b_1 \, x + v_1 \alpha \, x^2 $, to the vector field (\ref{linf}) we obtain
 \begin{equation}\label{linf2}
\phi_*(X)(u,v,w) \ig \lambda \, (u,v,w) \su (g'(v)(\lambda \, v + w),w, \alpha \, u^2) \, .
\end{equation}

\begin{teo} \label{teo1li}
 Consider a vector field $ X  \in \mathcal{N}_{li}(\lambda,3) $, with $ \lambda < 0 $, of the form (\ref{linf})  where $ g(t) = A_1 \, t + A_2 \, \frac{t^2}{2} $. Then $ X $ has
orbits that escape to infinity.
\end{teo}

\begin{proof} In the case $ A_2 = 0 $, making the linear change of coordinates
$$ (u,v,w) = \phi(x,y,z) = \frac{1}{m} (x, m y, m v_1 z) $$
where $ m = A_1 $, the vector field $ X  - \lambda \, I $ has the
form (\ref{li}). The result now follows from Theorem~\ref{liorsc}.

Next consider the case  $ A_2 \neq 0 $. We may assume

$$ X(x,y,z) = \lambda(x,y,z) \, + \, (g'(y)(\lambda y + z), z, v_1
\alpha x^2) \, . $$

\noindent To find orbits of $ X $ which escape to infinity, we
first make the coordinate change
$$ (u,v,w) \ig \frac{1}{z}(x,y,1) \; . $$
If $ Y $ is the vector field $ X $ in the new coordinates, then $ Z
 \ig w \, Y $ is defined by
$$ Z(u,v,w) \ig (-\beta u^3 + (A_1 \, w + A_2 \, v)(\lambda \, v + 1),
-\beta u^2 v + w , -w(\lambda w + \beta u^2))
 \,  $$
where $ \beta = v_1 \, \alpha $. For $ w \neq 0 $, the vector fields
$ Y $ and $ Z $ have the same orbits. Moreover,  for $ w > 0 $
(resp. $ w < 0 $), the orbits of $ Y $ and $ Z $ have the same
(resp. inverse) orientation. Now we apply the blow--up
$$ (s,q,p) = (u, \frac{v}{u^3},
\frac{w}{u^5}) \, . $$
If $ Y_1 $ is the vector field $ Y $ in the new
coordinates, then $ Y_1  \ig s^2 \, Z_1 $ where
$$ Z_1(s,q,p) \ig A(s,q,p) \, (s,-3q,-5p) \su (0,p - \beta q, -p(\beta + \lambda p  s^3))
 \, $$
and where $ A(s,q,p) \ig -\beta + (A_1 p s^2 + A_2 q)(\lambda q s^3
+ 1) $.

\noindent The singularities of $ Z_1 $ over $ s = 0 $ are
$$ (0,0,0) , \; \; (0, \frac{2
\beta}{3 A_2}, 0),\quad \textrm{and} \quad (0, \frac{4 \beta}{5 A_2}, \frac{8 \beta^2}{25 A_2}) \, . $$
The Jacobian matrix of $ Z_1 $ at $(0, \frac{4 \beta}{5 A_2},
\frac{8 \beta^2}{25 A_2})$ has  eigenvalues
$$ \mu_1 = -\frac{\beta}{5} \, , \; \;
\mu_2 = -\frac{2 \beta}{5} \, , \quad \textrm{and} \quad \mu_3 = -2 \beta \, . $$

\medskip

\noindent In the case $ \beta > 0 $ (resp.  $ \beta < 0 $), this
singularity is an attractor (resp. repeller) for vector field $ Z_1
$. Given  an initial condition $ (s(0), q(0), p(0)) $ sufficiently
close to the singularity, with $ s(0)p(0) > 0 $ (resp. $ s(0)p(0) <
0 $) for  $ \beta > 0 $ (resp. $ \beta < 0 $),  we obtain an orbit
of the original vector field $ X $  that  escapes positively to
infinity.

\end{proof}

\subsection{The discrete case}
In the discrete case, we prove that the answer to Question~2 is
negative for any $ g(t) \in \R[t] $ where $ g(0) = 0 $ and $ deg_t
g(t) \geq 1 $.

\medskip

For $ \ab{\lambda} < 1 $, consider
 \begin{eqnarray} \label{linfd}
 F(x,y,z) & = & \lambda \, (x,y,z) \su (0, v_1 \, z, 0) \su  \\
 & & g(t) \, (1, -(b_1 + 2v_1 \alpha \, x), \alpha \,  g(t))  \nonumber
 \end{eqnarray}
where $ t = y + b_1 \, x + v_1 \alpha \, x^2, v_1 \alpha \neq 0 $,
and $ g(t) \in \R[t] $ with $ g(0) = 0 $ and $ deg_t g(t) \geq 1 $.

\begin{lem} The set of fixed points of $ F $ is reduces to the origin.
\end{lem}
\begin{proof} If $ (x_0,y_0,z_0) $ is a fixed point of $ F $ and $ t_0 = y_0 + b_1 x_0 + v_1 \alpha x_0^2 $, then we have
\begin{eqnarray*}
(1 - \lambda) x_0 & = & g(t_0), \\
(1 - \lambda) z_0 & = & \alpha \, g(t_0)^2, \; \textrm{and} \\
 (1 - \lambda) y_0 & = & v_1 \alpha (1 - \lambda) \, x_0^2 \re (b_1 + 2v_1 \alpha \, x_0)(1 - \lambda) x_0 \\
& = & -b_1(1 - \lambda) \, x_0 - v_1 \alpha (1 - \lambda) \, x_0^2 \, .
\end{eqnarray*}
Therefore,
$$ t_0 = 0 \quad \Longrightarrow \quad g(t_0) = 0 \quad \Longrightarrow \quad (x_0,y_0,z_0) \ig (0,0,0) \, .$$
\end{proof}

\noindent {\bf Periodic points of period two.} \, Assume $
(x_0,y_0,z_0) \neq (0,0,0) $ is a periodic point of period two of $
F ,$ and let $ \beta = v_1 \alpha $. Then
\begin{eqnarray}
C_1(x_0,y_0,z_0) & = & (\lambda^2 - 1) \, x_0 + \lambda \, g(t_0) + g(t_1) \ig 0 \, , \label{cond1}\\
C_2(x_0,y_0,z_0) & = & (\lambda^2 - 1) \, z_0 + \lambda \, \alpha \,  g(t_0)^2 + \alpha \, g(t_1)^2 \ig 0 \, , \label{cond2} \\
C_3(x_0,y_0,z_0) & = & (\lambda^2 - 1) \, y_0 + 2 \lambda \, v_1 \, z_0 - \lambda \, (b_1 + 2 \beta \, x_0) \, g(t_0) + \beta \, g(t_0)^2 \nonumber \\
& & - (b_1 + 2 \beta \, (\lambda x_0 + g(t_0))) \, g(t_1) \ig 0 \,
\label{cond3}
\end{eqnarray}
where
\begin{eqnarray*}
t_0 & = & y_0 + b_1 x_0 + \beta \, x_0^2 \quad \textrm{and} \\
t_1 & = & \lambda \, b_1 \, x_0 + \lambda \, y_0 + v_1 \, z_0
- 2 \, \beta \, x_0 \, g(t_0) + \beta \, (\lambda x_0 + g(t_0))^2 \, .
\end{eqnarray*}
\begin{lem} If $ g(t) = t ,$ then the unique periodic point of period two of $ F $ is the origin.
\end{lem}
\begin{proof} Suppose $ F^2(x_0,y_0,z_0) = (x_0,y_0,z_0) $. In this
case, we have
$$ C_1(x_0,y_0,z_0) \ig  0 \; \Longrightarrow \; z_0 \ig \frac{1}{v_1} \, [(\lambda \beta + 1 - \lambda^2) \, x_0 - 2(\lambda - \beta) \, t_0 - \beta(\lambda \, x_0 + t_0)^2] \, .$$
Replacing this value of $ z_0 $ in the equation $ \frac{v_1}{1 +
\lambda} \, C_3(x_0,y_0,z_0) - C_2(x_0,y_0,z_0) = 0 ,$ we obtain
$$ (1 + \lambda)^2 \, (-x_0 + b_1 \, x_0 + \lambda \, x_0 + \beta \, x_0^2 + y_0) \ig 0 \, .$$
Replacing our $ z_0 $ and the value $ y_0 = (1 - \lambda - b_1) \,
x_0 - \beta \, x_0^2 $ in equations $(16)$ through $(18),$ we obtain
\begin{eqnarray*}
C_1(x_0,y_0,z_0) & = & 0 \, , \\
C_2(x_0,y_0,z_0) & = & -(\lambda - 1)^3 \, x_0 \ig 0 \, ,\\
C_3(x_0,y_0,z_0) & = & \frac{1}{v_1} \, (\lambda - 1)^3 \, (1 + \lambda) \, x_0 \ig 0 \, .
\end{eqnarray*}
This implies that $ (x_0,y_0,z_0) \ig (0,0,0) ,$ which completes the
proof.
\end{proof}

\bigskip

\noindent {\bf Periodic points of period three.} \, Assume $
(x_0,y_0,z_0) \neq (0,0,0) $ is a periodic point of period three of
$ F ,$ and let $ \beta = v_1 \alpha $. Then
\begin{eqnarray*}
D_1(x_0,y_0,z_0) & = & (\lambda^3 - 1) \, x_0 + \lambda^2 \, g(t_0) + \lambda \, g(t_1) + g(t_2) \ig 0 \, , \\
D_2(x_0,y_0,z_0) & = & (\lambda^3 - 1) \, z_0 + \lambda^2 \, \alpha \,  g(t_0)^2 + \lambda \, \alpha \, g(t_1)^2 + \alpha \, g(t_2)^2 \ig 0 \, ,\\
D_3(x_0,y_0,z_0) & = & (\lambda^3 - 1) \, y_0 + 3 \lambda^2 \, v_1 \, z_0 - \lambda^2 \, (b_1 + 2 \beta \, x_0) \, g(t_0) + 2 \lambda \, \beta \, g(t_0)^2 \\
& & - \lambda \, (b_1 + 2 \beta \, (\lambda x_0 + g(t_0))) \, g(t_1) + \beta \, g(t_1)^2 \\
& & -(b + 2\beta(\lambda^2 \, x_0 + \lambda \, g(t_0) + g(t_1))) \,
g(t_2) \ig 0 \,
\end{eqnarray*}
where
\begin{eqnarray*}
t_0 & = & y_0 + b_1 x_0 + \beta \, x_0^2 \quad \textrm{and} \\
t_1 & = & \lambda \, b_1 \, x_0 + \lambda \, y_0 + v_1 \, z_0
- 2 \, \beta \, x_0 \, g(t_0) + \beta \, (\lambda x_0 + g(t_0))^2 \, , \\
t_2 & = & b_1 \, \lambda^2 \, x_0 + \lambda^2 \, y_0 + 2 \lambda \, v_1 \, z_0   - 2\beta \, \lambda \, x_0 \, g(t_0) + \beta \, g(t_0)^2\\
& &   - 2\beta \, (\lambda \, x_0 + g(t_0)) \, g(t_1)
+ \beta \, (\lambda^2 \, x_0  + \lambda \, g(t_0) +  g(t_1))^2 \, .
\end{eqnarray*}

\begin{lem} If \, $ -1 < \lambda < 1 $ and $ \, g(t) = A \, t $, with $ A \neq 0 $, then $ F $
has a periodic point of period three $\, (x_0,y_0,z_0) \neq (0,0,0)
$. Furthermore, the eigenvalues of $ D F^3(x_0,y_0,z_0) $ are all
other than $1.$
\end{lem}
\begin{proof}
Computations were done using MATHEMATICA. These proved that the
point $ (x_0,y_0,z_0) ,$ where
\begin{eqnarray*}
x_0 & = & \frac{(1 + \lambda + \lambda^2) \, (1 + 4 \lambda^2 + \lambda^4)}{A \, \beta \, (1 - \lambda)^3} \; , \\
y_0 & = &  -\frac{1 + \lambda + \lambda^2}{A^2 \, \beta \, (1 - \lambda)^6} \, [\lambda \, (1 + \lambda + \lambda^2) \, (4 + \lambda + 8 \lambda^2 + 11 \lambda^3 + 4 \lambda^4 +
7 \lambda^5 + \lambda^7)  \\
& &  \su  A \, b_1 \, (1 - \lambda)^3 \, (1 + 4 \lambda^2 + \lambda^4)]  ,   \quad    {\textrm{and}}   \\
z_0 & = & \frac{(1 + \lambda + \lambda^2)^3 \, (1 + 3 \lambda^2 + 4 \lambda^3 + 3 \lambda^4 + \lambda^6)}{v_1 \, A^2 \, \beta \, (1 - \lambda)^5}
\end{eqnarray*}
is a periodic point of period three of $ F $. They also proved that
the characteristic polynomial of  $ D F^3(x_0,y_0,z_0) $ is
\begin{eqnarray*}
p(x) & = &  -\lambda^9
- \lambda \, (8 + 44 \, \lambda + 104 \, \lambda^2 + 164 \, \lambda^3 + 164 \, \lambda^4 +
     113 \, \lambda^5 + 44 \, \lambda^6  \\
& &  + 8 \, \lambda^7 - 4 \, \lambda^8) \, x +  (- 4 + 8 \, \lambda +
     44 \, \lambda^2 + 113 \, \lambda^3 + 164 \, \lambda^4 + 164 \, \lambda^5 \\
& & +
     104 \, \lambda^6 + 44 \, \lambda^7 + 8 \, \lambda^8) \, x^2 + x^3 \, ,
\end{eqnarray*}
and that
$$ p(1) = 3 \, (\lambda - 1)^3 \, (1 + \lambda + \lambda^2)^3 \neq 0 \, . $$

\end{proof}

\begin{teo}\label{teo3.9} For $ \ab{\lambda} < 1 $, consider
 \begin{eqnarray*}
 F(x,y,z) & = & \lambda \, (x,y,z) \su (0, v_1 \, z, 0) \su  \\
 & & g(t) \, (1, -(b_1 + 2v_1 \alpha \, x), \alpha \,  g(t))
 \end{eqnarray*}
where $ t = y + b_1 \, x + v_1 \alpha \, x^2,  \, v_1 \alpha \neq 0
$, and $ g(t) \in \R[t] ,$ with $ g(0) = 0 $ and $ g'(0) \neq 0 $.
Then there exists $ (x_0,y_0,z_0) \neq (0,0,0) $ which is a periodic
point of period 3 of $ F $.
\end{teo}
\begin{proof} Assume that $ g(t) = A \, t + A_2 \, t^2 + \dots + A_k \, t^k $, with $ A \neq 0 $. If $ (A_2,\dots,A_k) = (0,\dots,0) ,$
then $ F $ will be denoted $ F_0 $. Therefore, $ F_0 $ has a
periodic point of period three $ (x_0,y_0,z_0) \neq (0,0,0) $ and
the eigenvalues of $ D F_0^3(x_0,y_0,z_0) $ are other than from 1.
Consider the map $ G : \R^{k-1}\times \R^3 \to \R^3 $ defined by
$$ G(A_2,\dots,A_k,x,y,z) \ig F^3(x,y,z) \re (x,y,z) \, . $$
Note that $ G(0,\dots,0,x,y,z) = F_0^3(x,y,z) - (x,y,z) $, for all $
(x,y,z) \in \R^3 $. Then $ G(0,\dots,0,x_0,y_0,z_0) = (0,0,0) $ and
$ D_2 \, G(0,\dots,0,x_0,y_0,z_0) $ is invertible. By the Implicit
Function Theorem, there exists $ \varepsilon > 0 $ such that, for
all $ (A_2,\dots,A_k) $ with $ \max\{|A_2|,\dots,|A_k|\} <
\varepsilon ,$ the map $ F(A_2,\dots,A_k,.,.,.) $ has a periodic
point of period three. In the general case, note that if  $ a \in
\R-\{0\} $ and $ T(x,y,z) = a^{-1} \, (x,y,z) $, then
\begin{eqnarray*}
T(F(T^{-1}(u,v,w) & = & \lambda \, (u,v,w) \su (0, v_1 \, w, 0) \su  \\
 & & \tilde g(t) \, (1, -(b_1 + 2v_1 \tilde \alpha \, u), \tilde \alpha \,  \tilde g(t))
 \end{eqnarray*}
where $ \, \tilde \alpha = \, \alpha a \, $ and
$$ \tilde g(t) \ig  a^{-1} \, g(a \, t) \ig A \, t + A_2 \, a \, t^2  + \dots +
A_k \,  a^{k-1} \, t^k \, .$$ For $ |a| $ sufficiently small, the
map $ T \circ F \circ T^{-1} $ has a non--vanishing periodic point
of period three and, consequently, so does $F.$
\end{proof}


\begin{thebibliography}{99}

\bibitem[ChE]{ChE}
M. Chamberland, A. van den Essen, {\em Nilpotent Jacobians in
dimension three}, Journal of Pure and Applied Algebra {\bf 205}
(2006), 146--155.

\bibitem[CEGMH]{CEGMH}
A. Cima, A. van den Essen, A. Gasull, E. Hubbers and F. Ma\~nosas,
{\em A Polynomial Counterexample to the Markus--Yamabe Conjecture},
Advances in Mathematics {\bf 131} (1997), 453--457.


\bibitem[CGM1]{CGM1}
A. Cima, A. Gasull, F. Ma\~nosas, {\em A Polynomial Class of
Markus--Yamabe Counterexamples}, Publicacions Matem\`atiques {\bf
41} (1997), 85--100.

\bibitem[CGM2]{CGM2}
A. Cima, A. Gasull, F. Ma\~nosas, {\em The discrete
Markus--Yamabe problem}, Nonlinear Anal. {\bf
35} (1999), 343--354.

\bibitem[vE]{vE}
A. van den Essen, {\em  Polynomial Automorphisms and the Jacobian
Conjecture}, Progress in Mathematics, vol. 190, Birkhauser, Basel,
2000.

\bibitem[Gui-Cast]{GC}
 V. Gu\'{\i}\~nez and A. Casta\~neda, {\em A Polynomial Class of Markus--Yamabe Counterexamples and Examples in $\R^3$}, Applicable Analysis
{\bf 90} (2011), 787--798.



\bibitem[Gu]{Gu}
C. Gutierrez, {\em A solution to the bidimensional Global Asymptotic
Stability Conjecture}, Ann. Inst. H. Poincar\'e Anal. Non Lin\'eaire
{\bf 12} (1995), 627–-671.

 \bibitem[MY]{MY}
L. Markus,  H. Yamabe, {\em Global Stability Criteria for
Differential Systems}, Osaka Math. J. {\bf 12} (1960), 305–-317.




\end{thebibliography}
\end{document}